\documentclass[12pt,a4paper]{amsart}
\usepackage{preamble}
\usepackage{xcolor}
\title[Hyperbolic Optimal Response]{Optimal Response for Hyperbolic Systems by the fast adjoint response method}
\date{March 2024}

\begin{document}
\begin{abstract}
In a uniformly hyperbolic system, we consider the problem of finding the optimal infinitesimal perturbation to apply to the system, from a certain set $P$ of feasible ones, to maximally increase the expectation of a given observation function.

We perturb the system both by composing with a diffeomorphism near the identity or by adding a deterministic perturbation to the dynamics.
In both cases, using the fast adjoint response formula, we show that the linear response operator,
which associates the response of the expectation to the perturbation on the dynamics,
is bounded in terms of the $C^{1,\alpha}$ norm of the perturbation. 
Under the assumption that $P$ is a strictly convex, closed subset of a Hilbert space $\cH$ that can be continuously mapped in the space of $C^3$ vector fields on our phase space,
we show that there is a unique optimal perturbation in $P$ that maximizes the increase of the given observation function. Furthermore since the response operator is represented by a certain element $v$ of $\cH$,  when the feasible set $P$ is the unit ball of $\cH$, the optimal perturbation is $v/||v||_{\cH}$.
We also show how to compute the Fourier expansion $v$ in different cases.
Our approach can work even on high dimensional systems. We demonstrate our method on numerical examples in dimensions 2, 3, and 21.

\end{abstract}

\maketitle

\section{Introduction}

\subsection{Literature review}
\hfill\vspace{0.1in}
\label{s:literature review}

The understanding of the statistical properties of the long term behaviour of deterministic dynamical systems has important applications in the natural and social  sciences. 
When a dynamical system is perturbed, it is useful to understand and predict the response of the system's statistical properties to the perturbations. 
We say that the system exhibits a linear response to the perturbation when such a response is differentiable. 
In this case, the long-time average of a given observable  changes smoothly during the perturbation (see \Cref{respdef} for a more precise definition). 
Understanding the linear response has particular importance in applications, in particular in climate science (see e.g. \cite{GhLuc}).

Uniformly hyperbolic maps display a linear response for suitably smooth perturbations and observables. This was first proved by Ruelle in \cite{Ruelle_diff_maps} and then generalized to many directions in further works (see \cite{Gouezel2006} or \cite{Baladi2007} for a survey).
There are several  formulas and characterizations for the linear response of hyperbolic systems, but many involve distributions or unstable parts. 
We will use a so-called `fast adjoint response' linear response formula (see \cite{Ni_asl,TrsfOprt}), which is based on a pointwise defined function, and on which it is thus easy to perform analytic estimates and numerical computations.



In the present paper, we address a natural inverse problem related to the Linear Response: the Optimal Response. 
Given an observable function $\Phi$, we study which infinitesimal perturbation, chosen from a set of feasible perturbations $P$, produces the greatest change in the expectation of $\Phi$. 
In the paper, we address this problem in the case where the system is uniformly hyperbolic, finding conditions under which the optimal perturbation uniquely exists, and we show a method to approximate numerically the optimal perturbation.

The optimal response problem was studied for the first time in \cite{ADF}, for finite state Markov chains.
The case of dynamical systems having kernel transfer operators (including random dynamical systems having additive noise) was considered in \cite{optimalresponse2022}.
The case of 1-dimensional deterministic expanding circle maps was considered in \cite{optimal_response_23}.
Those papers also considered the problem of optimizing the spectral gap and hence the speed of mixing (see also \cite{Froyland2007}).
The paper \cite{G22} considers a similar problem, where the optimal coupling is studied in the context of mean field coupled systems.

A related problem is the `linear request problem', related to the search for a perturbation achieving a prescribed response direction \cite{Bodai20,optimalresponse2017,GG19,MacKay18,Kloeckner18}.
All these studies are related to the optimal response, in terms of understanding how a system can be modified in order to control the behavior of its statistical properties.

The method we use for the numerical approximation of the optimal perturbation, which is an improvement of the method presented in \cite{optimal_response_23}, is based on the computation of the responses of the elements from a Hilbert basis of the space of allowed perturbations (see \Cref{s:fourier}).
Computing the linear response of a system to a given perturbation is a nontrivial task which was considered in different works, with different purposes and approaches. Some approaches can  approximate the response up to explicit small errors (see e.g. \cite{Bahsoun2018,Pollicott2016,NiTC}), while others are only proven to converge to the correct value under suitable assumptions.
In this direction, computing the response of multidimensional hyperbolic maps is a hard task which  was approached in relatively recent times by several works (see e.g. \cite{fr,far,Chandramoorthy2021a}). 
The present paper applies the fast adjoint response algorithm developed in \cite{far}.
This algorithm is based on a characterization of the response formula which is based on functions which can be stably computed on orbits of the system (see Section \ref{s:review far}). 
This method is hence reliable and allows fast computations even in high dimensions, also allowing us to compute the linear response of many perturbations as we need to compute the optimal perturbation.

\subsection{Main results}\label{1.2}
\hfill\vspace{0.1in}

In this section we  explain more precisely but still a bit informally the kind of problems we are going to consider and outline our main results. 

Let us consider  a mixing axiom A attractor, $K$, of a $C^3$ diffeomorphism $f$, on a $C^\infty$ manifold $\cM$ of dimension $M$. 
We denote its physical measure  by $\mu$, supported on $K$. 
We consider a certain observable $\Phi:\cM\to \mathbb{R}$ whose expectation we want to increase as much as possible, by applying perturbations from a certain set of allowed perturbations $P$.
It is natural to think of the set of allowed perturbations $P$ as a convex set because if two different perturbations of the system are possible, then their convex combination (applying the two perturbations with different intensities) should also be possible. 
We will also consider $P$ as a subset of a suitable separable Hilbert space $\cH$ which is supposed to be continuously mapped on the a space of $C^3$ vector fields on a suitable neighborhood of the attractor of $f$ (an example is when $\cH$ is is a Sobolev space $H^p$ for $p$ large enough and hence is a subspace of $C^3$).
We will consider perturbations to the system both by composing with a diffeomorphism near to the identity or by perturbing directly the map defining the dynamics by an additive deterministic perturbation. We will see that for both kinds of perturbation we have the same properties and results, thus here we will generically indicate with $X $ a vector field indicating the direction of perturbation in one of the two ways to perturb a system we consider
(see the beginning of Section \ref{s:review far} for the precise definition of the  perturbations by composing with a diffeomorphism and Section \ref{S:add} for the details about the additive deterministic perturbation to the map). 
Let $X\in P$ be such a vector field, representing a feasible direction of perturbation. 
Let $\mu _{X,\gamma }$ \ be the invariant probability measure of the system after applying a perturbation in the direction $X$ with intensity $\gamma $ (see Sections \ref{s:review far} and \ref{S:add} for the details on how the response operators are defined). 
Let the response of the average of $\Phi$ to this perturbation be denoted as%
\begin{equation} \label{respdef}
R(X):=\lim_{\gamma \rightarrow 0}\frac{\int \Phi \ d\mu _{X,\gamma }-\int \Phi  \ d\mu}  {%
\gamma }.
\end{equation}%
We are interested in finding the element $X\in P$ for which $R(X)$ is maximized, thus we are interested in finding $X_{opt}$ such that
\begin{equation} \label{e:maxprob}
R(X_{opt}) = \underset{X\in P}{\sup }R(X). 
\end{equation}
As we will see, this is the maximization of a continuous linear function on a convex set, and the maximum can be achieved.

The first result of this paper (see  \Cref{l:boundR}) shows that, for a  hyperbolic system $(\cM, f)$, there is $0<\alpha<1$, such that the linear response operator $R$ can be extended to the space of $C^{1,\alpha}$ vector fields and $R: C^{1,\alpha}(\cM, T\cM) \rightarrow \R$ is linear and bounded.
In the proof of \Cref{l:boundR} we use the recent characterization of the linear response formula of hyperbolic maps given in \cite{Ni_asl} and \cite{TrsfOprt} (see Section \ref{s:review far}). This characterization  enables a simple estimation of the linear response in the hyperbolic case, only requiring fast enough decay of correlations.
In particular we have that $R$ is also continuous on the space of $C^3$ vector fields, and as $\cH$ is continuously mapped in such space we can see this operator as a continuous linear operator on $\cH$.
This easily implies that when $P$ is strictly convex and $R$ is not identically null the problem \eqref{e:maxprob} has a unique solution (see Section \ref{sec:gen} and Proposition \ref{gen:ex} for precise statements).

Once we know that $R$ is a continuous linear operator on the Hilbert space $\cH$, by the Riesz representation theorem, there is a unique $\cH$ representative of $R$, denoted by $v$.
By a simple calculation, we can show that $X_{opt} = v/||v||_{\cH}$ is the unique maximizer when the permissible set $P$ is the unit ball in $\cH$.
It is then easy to approximate $v$ by computing its Fourier coefficients by applying $R$ on a Fourier basis of $\cH$  (see Proposition \ref{gen:algo}).

When the background space $\cM$ is a torus, we can choose $\cH$ as the Sobolev space $H^p(\cM, T\cM)$ and there is an obvious orthogonal basis of $H^p(\cM, T\cM)$, which is the set of trigonometric functions.
We can then compute the Fourier expansion of $v$ and $X_{opt}$ by applying the response operator $R$ to each element of the Hilbert basis (see Section \ref{s:fourier}).

In \Cref{s:numeric} we numerically implement the algorithm explained above on examples of hyperbolic maps on the 2, 3, and 21-dimensional tori; the results confirm that our method finds the optimal response, and that the linear responses we compute are correct, also showing that our method is suitable to address the problem on  high-dimensional systems.

\Cref{s:2 general cases} generalizes the approximation scheme for the computation of the optimal perturbation, giving a recipe to approximate $v$ when $\cM$ is not a torus, and there is no obvious orthogonal basis, so the Fourier method is not easy to implement.
We relate the problem to solving a Laplacian equation with suitable boundary conditions.
We will not perform numerical investigations in this case.

\section{Review on fast response formulas}
\label{s:review far}

%

In the following, we will benefit from a particular characterization of the linear response for deterministic perturbations of uniformly hyperbolic systems established in \cite{Ni_asl,TrsfOprt}.
This characterization is called the `fast adjoint response formula'.  
In this section we review the formula and the tools necessary to understand it. 
The formula is composed of two parts, the adjoint shadowing lemma \cite{Ni_asl} and the equivariant divergence formula \cite{TrsfOprt}.
The continuous-time versions of fast response formulas can be found in \cite{Ni_asl,vdivF}.

Let us consider a dynamical system $(\cM, f)$ where $\cM$ is a  $C^\infty$ manifold  of dimension $M$ and $f:\cM \to \cM$ is a  {$C^3$} diffeomorphism.
Suppose the system has  a mixing axiom A attractor we  denote by $K$ and on which a physical measure $\mu$ is supported. 
Denote the stable and unstable subspaces by $V^s$ and $V^u$, and the pushforward operator on vectors by $f_*$.
Consider a vector field $X$ on $\cM$.
We define the oblique decomposition of $X$ as
\[ \begin{split}
  X = X^u + X^s, \quad \textnormal{where} \quad X^u\in V^u,  \; X^s \in V^s.
\end{split} \]

In this section we will perturb $f$ by composing it by a family of $C^3$ diffeomorphinms $\tf_\gamma$  smoothly parameterized by a small real number $\gamma$, which converges to the identity at  $\gamma=0$.
We then consider the family of systems $(\cM, f_\gamma)$ where 
\[
f_\gamma = \tf_{\gamma} \circ f
\]
whose physical measures will be denoted by $\mu_\gamma$.
The existence of the linear response in this case was proved when the map $\gamma\mapsto \tf_\gamma$ is $C^1$ from $\R$ to the family of {$C^3$}  diffeomorphisms on $\cM$.
In this paper, the regularities of all functions and maps are stated on the entire $\cM$ unless otherwise noted.

Denote the derivative with respect to the parameter $\gamma$ as 
$\delta (\cdot) :=\left. \pp{(\cdot)}\gamma \right|_{\gamma=0}$, suppose the perturbation on the dynamics is such that 
\begin{equation} \label{e:X}
\delta \tf_{\gamma} :=\left. \pp{\tf_{\gamma}}\gamma \right| _{\gamma=0} = X.
\end{equation}
In the following, for optimization purposes we will assume that $X$ has a further regularity and belongs to a suitable Sobolev space $X\in H^p(\cM)$ with $p$ such that $H^p\subset C^3$ (see \Cref{e:morrey}).

For any fixed observable function $\Phi\in C^3(\cM)$, the linear response $\delta  \mu_\gamma(\Phi)$   has the following decomposition into what we call the `shadowing contribution' and `unstable contribution' \cite{Ruelle_diff_maps}.
\[ \begin{split}
\delta  \mu_\gamma(\Phi) 
  = S.C. + U.C. ,
  \quad\quad
  S.C. = \mu (d\Phi \cdot S(X)) , \\
  U.C. = \lim_{W\rightarrow\infty} \mu \left(\phi_W \frac {\delta \tL^u_\gamma \sigma} {\sigma}\right),
  \quad \textnormal{where} \quad 
  \phi_W:= \sum_{m=-W}^W \left[\Phi\circ f^m 
  - \mu(\Phi) \right].
\end{split} \]
Here $\mu$ is the physical measure for the map $f$.
We will explain the notations and the variant version of the linear response formula that we use in the following paragraphs.

In the shadowing contribution, $S:\cX^\alpha \rightarrow \cX^\alpha$ is the (linearized) shadowing operator on $\cX^\alpha$, the space of H\"{o}lder continuous vector fields.
A characterization for $S$ is that, $v := S(X) $ is the only bounded vector field satisfying the variation equation 
\[v = f_*v + X.\]
Intuitively, since the variation equation describes the evolution of small perturbations of an orbit, the boundedness of $v$ indicates that $v$ is the difference between two shadowing orbits, which are two close orbits with governing equations perturbed in the direction of $X$.

For the shadowing contribution, we can apply the adjoint shadowing lemma in \cite{Ni_asl}, to get
\[ \begin{split}
  S.C. = \mu (\cS(d\Phi) X),
\end{split} \]
where $\cS:\cX^{*\alpha} \rightarrow \cX^{*\alpha}$ is the adjoint shadowing operator on H\"{o}lder continuous covector fields.
A characterization for $\cS$ is that, $\omega := \cS(d\Phi) $ is the only bounded covector field such that 
\[\omega = f^*\omega + d\Phi.\]
In \cite{Ni_asl}, we further showed that $\cS(d\Phi)$ has two other equivalent characterizations, and it is H\"{o}lder continuous.
The computation of $\cS$ and $S$ are given by the nonintrusive shadowing algorithm in \cite{Ni_nilsas,Ni_NILSS_JCP}.

In the unstable contribution, $\sigma$ is the density of the conditional measure of $\mu$ on unstable foliations.
As explained in detail in \cite{TrsfOprt}, we define $\tL^u_\gamma$ as the transfer operator of the composition of $\tf_\gamma$ and the holonomy map.
As an intuitive explanation, $\tL_\gamma^u$ is the transfer operator redistributing the probability on unstable submanifolds according to the vector field $\gamma X^u$, and
\[
\delta \tL_\gamma^u \sigma := \left. \pp{\tL^u_\gamma \sigma}{\gamma} \right| _{\gamma=0}.
\]
We could write this as a submanifold divergence, $\delta \tL^u_\gamma \sigma= \div^u (\sigma X^u)$, but $X^u$ is not differentiable, so this expression is not useful for our purposes.

The equivariant divergence formula in \cite{TrsfOprt} is a pointwisely defined expression (without exponentially growing terms) for the term $\frac{\delta \tL^u_\gamma \sigma}{\sigma}$ in the unstable part,
\begin{equation}\begin{split}
\label{e:uto}
  - \frac{\delta \tL^u_\gamma \sigma}{\sigma}
  = \div^v X + (\cS(\div^vf_*)) X.
\end{split} \end{equation}
Here $\cS$ is the adjoint shadowing operator; $\div^v X$ is $\nabla X$ contracted by the unstable hypercube and its co-hypercube in the adjoint unstable subspace, 
\begin{equation}\begin{split}
\label{e:vdiv}
  \div^v X:= \teps \nabla_\te X.
\end{split} \end{equation}
Here $\tilde e = e_1\wcw e_u$ is the unit unstable $u$-vector {field} defined on the attractor {of $f$},
where $e_i\in V^u$ are unstable vectors, and $\teps$ is the unstable $u$-covector field such that $\teps(\te)=1$ everywhere on the attractor.
Note that $\te$ and $\teps$ are not differentiated in this expression.
Similarly, $\div^v f_*$ is the contraction of the 2-tensor $f_*$ by the same dual basis, so $\div^v f_*$ is a H\"{o}lder-continuous covector field.

We considered a family $f_\delta$ of perturbations of $f$ with physical invariant measures $\mu_\delta$. Since it is well known (\cite{Ruelle_diff_maps}) that the linear response $\delta  \mu_{\gamma}(\Phi)$  depends only on the direction of perturbation $X$.  We can now define more precisely  the linear response operator  $\mathcal{R}: C^3(\cM,T\cM)\to \mathbb{R}$, outlined in \eqref{respdef} as 
\begin{equation}\label{Resp}
{\mathcal{R}}(X):=
\delta  \mu_{\gamma}(\Phi).
\end{equation}

Summarizing the above formulas from \cite{Ni_asl} and \cite{TrsfOprt}, we have the following expression for the linear response  
\[ \begin{split}
{\mathcal{R}}(X)   = \lim_{W\rightarrow\infty} 
  \mu \left[(\cS(d\Phi) + \phi_W \cS(\div^vf_*) )  X 
  + \phi_W  \div^v X \right] .
\end{split} \]
We call this characterization of the linear response the {\it fast adjoint response formula.}
To estimate $R$, we decompose the linear response into three parts,   
\begin{equation} \label{R1} \begin{split} 
R_1(X) 
:= \mu \left[\cS(d\Phi) X \right].
\end{split} \end{equation}
\begin{equation}\label{R2W} \begin{split}
R_{2W}(X)
  := \mu \left[ \phi_W (\cS(\div^vf_*) X )\right].
\end{split} \end{equation}
\begin{equation} \label{R3W}\begin{split}
R_{3W}(X)
  := \mu \left[ \phi_W \div^v X \right].
\end{split} \end{equation}
So $\mathcal{R} = \lim_{W\rightarrow\infty} R_1 + R_{2W} + R_{3W}$.

In numerical experiments, since the fast adjoint response formula transforms everything into suitable functions, it is particularly suitable for computational purposes and orbit-based methods.
The fast adjoint response formula can be computed by evolving only $2u$ vectors per step, which is currently the most efficient for orbit-based sampling in higher dimensions \cite{far,fr}.

Moreover, the algorithm based on such formula is `adjoint' in the utility sense, which means that it tends to be faster when we want to compute the linear responses of many perturbations.
The earlier fast `tangent' response formula in \cite{fr} does not involve covectors, so it runs only forward in time, and is more efficient if we only want to compute the linear responses of only a few perturbations. 
The fast response algorithms are currently the most efficient for sampling the hyperbolic linear response by an orbit.

\section{Continuity and extension of the Linear Response operator}\label{sec:cont}
In this section we prove that the response operator $\mathcal{R}$ well-defined in \eqref{Resp} is defined and continuous on the space of $C^{1+\alpha}$ vector fields on $\cM$, and hence on the space of $C^3$ vector fields.
This regularity result is the basis for the study of the optimal perturbations and response of the present work.

\begin{lemma}
\label{l:boundR}
There is $0<\alpha<1$ and $C>0$ such that,
\[
|\mathcal{R}(X)|
\le C || X ||_{C^{1,\alpha}}
\le || X ||_{C^3}.
\]
Hence, $\mathcal{R}$ is a bounded operator on $C^3(\cM,T\cM)$.
\end{lemma}

\begin{remark*}
The first inequality still holds for the norm $|| X ||_{C^{1,\alpha}(K)}$, which is the $C^{1,\alpha}$ morm measured on the attractor $K$ instead of the entire $\cM$.
This is because the fast adjoint response formulas involve only quantities on the attractor. 
\end{remark*}

\begin{proof} We use the characterization of the response operator shown at \eqref{R1}, \eqref{R2W}, \eqref{R3W} and estimate each one of the summands  $R_1$, $R_{2W}$, $R_{3W}$ described there.
For for the first summand $R_1$ we have,
\[ \begin{split}
|R_1(X) |
= |\mu \left[\cS(d\Phi) X \right]|
\le C ||X||_{C^0},
\end{split} \]
where the $C^0$ norm is just the maximum of the absolute value of the function, and the constant $C$ does not depend on $X$ (but may depend on the unperturbed dynamics).

For $R_{2W}$, first recall that both $V^u$ and $V^{u*}$ are H\"{o}lder continuous, and that the expression of the equivariant divergence $\div^v$ in \Cref{e:vdiv} indicates that $\div ^v f_*$ is a H\"{o}lder continuous covector field on the attractor.
Since the adjoint shadowing operator $\cS$ preserves H\"{o}lder continuity \cite[appendix]{Ni_asl}, there is a constant $\alpha>0$, such that $\cS (\div ^v f_*) $ is a $C^{\alpha}$ covector field.
By the decay of correlation in mixing axiom A systems (with respect to H\"{o}lder observables, see for example \cite{DKL}, Theorem 4.1) there is a constant $C>0$ and $0<\lambda<1$ such that

\[ \begin{split}
|R_{2W}(X)|
= |\mu \left[ \phi_W( \cS(\div^vf_*) X )\right]|
\le \sum _{n=-W}^W C \lambda^n 
||\cS (\div ^v f_*) X||_{C^\alpha}
||{\color{red}}\Phi||_{C^1}.
\end{split} \]
We want to bound $||\cS (\div ^v f_*) X ||_{C^\alpha}$ by some norm of $X$.
Denote $\nu:=\cS (\div ^v f_*)$, which is a $C^\alpha$ covector field.
For any $x, y$ in the compact attractor,
by the `finite difference' version of the product rule for differentiation,
\[\begin{split}
\frac{|\nu(x) X (x) - \nu(y) X (y)|}{|x-y|^\alpha}
\le 
\frac{|\nu(x) X (x) - \nu(x) X (y)|}{|x-y|^\alpha}
+\frac{|\nu(x) X (y) - \nu(y) X (y)|}{|x-y|^\alpha}
\end{split}\]
Since $|x-y| < C |x-y|^\alpha$ on the compact set $\cM$,
\[\begin{split}
\frac{|\nu(x) X (x) - \nu(y) X (y)|}{|x-y|^\alpha}
\le 
C ||\nu||_{C^0} ||X||_{C^1} + C ||\nu||_{C^\alpha} ||X||_{C^0}    
\le 2C
||\nu||_{C^\alpha} ||X||_{C^1}.
\end{split}\]
Hence, $
||\cS (\div ^v f_*) X ||_{C^\alpha}
\le 2 C ||\cS (\div ^v f_*)||_{C^\alpha} ||X||_{C^1}$,
and we can define the limit 
\[
R_{2} := \lim_{W\rightarrow\infty} R_{2W}.
\]
Hence, we have the bound
\[
|R_2| \le C ||X||_{C^1}
\]
for some $C\geq0$.

For $R_{3W}$, note that $\div^v X$ defined in \Cref{e:vdiv} is a contraction of $\nabla X$, so for fixed $x$ and then for any $y$ close to $x$, in local coordinates, we have 
\[
|\div^v X(x)-\div^v X(y)|
= |tr (\ueps \nabla X \ue (x) )- 
tr (\ueps \nabla X \ue (y) )|.
\]
Here $\ue$ is a matrix whose columns are the basis elements of $V^u$, and $\ueps$ is the matrix composed of the dual basis such that $\ueps^T e = I_{u\times u}$.
Note that $tr$, the trace function of matrices, would return the same result regardless of the selection of the basis, as long as we are looking at the same subspace.
Since $V^u$ and $V^{u*}$ are H\"{o}lder, we can find basis at $y$ such that 
\[
|\ue(x) - \ue(y)|, \;
|\ueps(x) - \ueps(y)|
\le C |x-y|^\alpha
\]
Here the matrix norm can be taken as the max absolute entry value.
Hence, apply the product rule to the expression of $\div^v X$, we get
\[\begin{split}
|\div^v X(x)-\div^v X(y)|
\le C|\ueps \nabla X \ue (x) - 
\ueps \nabla X \ue (y) |
\le C|\ueps(x) \nabla X(x) ( \ue(x)- \ue (y)) | 
\\
+ C|\ueps(x) (\nabla X(x) -\nabla X(y)) \ue (y)) | 
+ C|(\ueps(x) -\ueps(y) \nabla X(y) \ue (y)) | 
\\
\le C  ||\nabla X||_{C^0} |x-y|^\alpha
+ C ||\nabla X||_{C^\alpha} |x-y|^\alpha
+ C |x-y|^\alpha ||\nabla X||_{C^0}
\\
\le  C ||\nabla X||_{C^{\alpha}} |x-y|^\alpha
\le  C ||X||_{C^{1,\alpha}} |x-y|^\alpha.
\end{split}
\]
So $\div^v X$ is H\"{o}lder with norm $||\div ^v X ||_{C^\alpha}\le C ||X||_{C^{1,\alpha}}$.
Hence,
\[ \begin{split}
|R_{3W}(X)|
:=
|\mu \left[ \phi_W \div^v X \right]|
\le \sum _{n=1}^W C \lambda^n 
||\div ^v X ||_{C^\alpha}
||\Phi||_{C^1}
\le C || X ||_{C^{1,\alpha}}
\end{split} \]
for some $C\geq0$.
\end{proof}

\begin{proposition}
\label{l:Rextend}
For any fixed observable function $\Phi\in C^3$, and the $\alpha$ prescribed in \cref{l:boundR}, there is a unique continuous extension of the definition domain of the linear response operator $\mathcal{R}$ to the entire $C^{1,\alpha}$. 
\end{proposition}

\begin{remark*}
We extend the definition of $\mathcal{R}$ to $C^{1,\alpha}$, but we did not prove that the linear response exists for $\delta \tf \in C^{1,\alpha}$, which is more difficult.
It seems plausible that the perturbation $\delta \tf$ may be less regular than the initial map $f$. 
Similar observations for expanding maps have been made and proved in \cite{optimal_response_23}.
\end{remark*}

\begin{proof}
For any $X\in C^{1,\alpha}$, since $C^3$ is dense in $C^{1,\alpha}$, we can find $\{X_k\}_{k\in \N} \subset C^3$ such that $X_k \rightarrow X$ in $C^{1,\alpha}$.
Hence, $\{X_k\}_{k\in \N}$ is a Cauchy sequence in $C^{1,\alpha}$, that is, 
\[
\lim_{k\rightarrow\infty} \sup_{m,n \ge k}
||X_m - X_n||_{C^{1,\alpha}} = 0
\]

We already have $\mathcal{R}$ defined on $C^3$, due to proofs in for example \cite{Ruelle_diff_maps}.
Since $R$ is linear, \cref{l:boundR} shows that
\[
|\mathcal{R}(X_m) - \mathcal{R}(X_n)|
= |\mathcal{R}(X_m - X_n)|
\le ||X_m - X_n||_{C^{1,\alpha}}.
\]
Hence, $\mathcal{R}(X_k)$ is a Cauchy sequence in $\R$, so we can define $\mathcal{R}(X)$ 
\[
\mathcal{R}(X) :=\lim_{k\rightarrow\infty} \mathcal{R}(X_k).
\]

This limit does not depend on the selection of the sequence $\{X_k\}_{k\in \N}$.
Indeed, if another sequence $Y_k \rightarrow X$ in $C^{1,\alpha}$,
then we can repeat the previous proof for the sequence $X_1, Y_1, X_2, Y_2, \ldots$
\end{proof}

\section{Existence and characterization of optimal perturbation}
\label{s:exist and char}

This section establishes the unique existence of the optimal response under suitable general assumptions also showing how to compute its Fourier expansion.
This is explained in Subsection \ref{sec:gen}.

We then show in \Cref{s:Hp rep} more concrete examples where the perturbations are given by the composition with a diffeomorphism close to the identity and their direction range  in a Sobolev space $H^p$ for a suitable $p$.
In \Cref{s:fourier} we consider the above kinds of perturbations and we show how to approximate the optimal perturbation in the case where $\cM$ is the $M-$dimensional torus.

\Cref{S:add} Considers the same problem for additive deterministic perturbations to $f$ and for these slightly different perturbations, results similar to the ones proved in \Cref{s:Hp rep} are established.

\subsection{A general setup for the optimal response.}\label{sec:gen}
\hfill\vspace{0.1in}

In this section we set up the problem in a general framework, which will allow us to consider different kinds of perturbations in the following.
We remark that in the applications it is reasonable to think that the perturbations we  apply to the system  are somewhat independent from the system, for example we might have a hight dimensional phase space, but a restricted set of perturbations we can apply to the system. For example, one can have control only on some parameter defining the system  or act on the dynamics only  along some coordinate of the phase space.  
\footnote{We will see in the following (see the end of Section \ref{s:fourier}) that the computational cost of the algorithm we propose to approximate the optimal perturbation is mainly related to the "size" of the set where the perturbations can range rather than the dimension of the phase space $\cM$. This makes the algorithm particularly suitable for applications where $\cM$ is high dimensional.}

We will hence suppose that the abstract perturbations we may apply range on a certain subset $P$ of a separable Hilbert space $\mathcal{H}$  endowed with a scalar product $\langle ,  \rangle$ and norm $||.||_{\mathcal{H}}$.
We will suppose that there is a linear and continuous function $I: \mathcal{H}\to C^3(\cM,T\cM)$  associating to an abstract perturbation, a concrete vector field on $\cM$ characterizing a direction of perturbation.
The associated response to the abstract perturbation is $R:P\to \mathbb{R}$, defined by
\begin{equation}\label{Resp2}
    R(p):=\mathcal{R}(I(p)).
\end{equation}

We remark that since $I$ is continuous, by Lemma \ref{l:boundR}, $R$ is continuous.
In this framework, the problem we consider is to find $X_{opt}\in P$ such that 
\begin{equation}
\label{e:gen-func-opt-prob}
{R}(X_{opt}) = \max_{X\in P}{R}(X).
\end{equation}

This is a general formalization  of the informal problem \eqref{e:maxprob} stated in Section \ref{1.2}.

Now we show a general result, establishing the existence and the uniqueness of the optimal perturbation in this setting.

\begin{definition}[Strictly convex set]
\label{stconv}We say that a convex closed set $A\subseteq \mathcal{H}$ is 
{strictly convex} if for all pair $x,y\in A$ and for all $0<\gamma <1$,
the points $\gamma x+(1-\gamma )y\in \mathrm{int}(A)$, where the relative
interior\footnote{%
The relative interior of a closed convex set $C$ is the interior of $C$
relative to the closed affine hull of $C$.} is meant.
\end{definition}
{Note that the unit ball of $\mathcal{H}$ is strictly convex.} With this definition we can state

\begin{proposition}\label{gen:ex}
    Let $P$ be a bounded, convex and closed set of $\mathcal{H}$, then the problem \eqref{opt-prob-add} has a solution. 
    If furthermore $P$ is strictly convex and $R$ is not constanttly null on $P$, then the solution is unique.
\end{proposition}

\begin{proof}Since $I$ is linear and continuous, then $R$ is also linear and continuous. We are hence optimizing a bounded linear function on a closed bounded convex set of an Hilbert space.  The statement hence follows from the general theory  on the optimization of a linear function on a convex closed set (see Proposition 4.1, 4.3 in \cite{optimalresponse2022} for more details).
\end{proof}

We now see that when $P$ is the unit ball of $\mathcal{H}$ there is a general way to compute  an approximation of the optimal perturbation $X_{opt}$ by computing its Fourier coefficients with respect to an orthonormal Fourier basis.

\begin{proposition}\label{gen:algo}
    Let  \(\{b_i\}\) be {a Fourier} basis of the Hilbert space 
    \(\mathcal{H}\). In the case $R:\mathcal{H}\to \mathbb{R}$ is not constantly null, the optimization problem \ref{e:gen-func-opt-prob} has unique solution $X_{opt}=\frac{v}{||v||_{\mathcal{H}}}$ where the {Fourier} coefficients $c_i$ of $v\in \mathcal{H}$ are characterized by  $$c_i:=\langle v , b_i \rangle=R(b_i).$$
\end{proposition}

\begin{proof}

 By the Riesz Representation Theorem, the linear and continuous functional \(R: \mathcal{H} \to \mathbb{R}\) has a representative \(v \in \mathcal{H}\), such that:
   \[
   R(w) = \langle w, v \rangle \quad \text{for all } w \in H.
   \]
The maximum $R(w)$  of \eqref{e:gen-func-opt-prob} is hence realized at $\frac{v}{||v||_{\mathcal{H}}}$ and
$ R(b_i) = \langle b_i, v \rangle$, characterizing the Fourier coefficients of $v$.
\end{proof}

We remark that the last proposition shows that the optimal perturbation can be approximated by computing the linear response $ R(b_i)$ on the elements of the basis. This is a simplification and a generalization of the method used in \cite{optimal_response_23} where the Fourier approximation is used together with Lagrange multipliers to compute the optimal perturbation for  expanding maps of the circle.

\subsection{The optimal response problem for perturbations by diffeomorphisms near to the identity.} 
\hfill\vspace{0.1in}
\label{s:Hp rep}

Let us consider again an axiom A system on a manifold $\cM$ having dimension $M$ with a compact mixing attractor, and suppose the perturbations we apply on our system are by composition with a diffeomorphism near to the identity whose direction is represented by a vector field $X$ as in \eqref{e:X}. The operator $\mathcal{R}$ is then defined by \eqref{Resp} and Lemma \ref{l:boundR} proves that $\mathcal{R}: C^3\to \R$  is a bounded linear operator.
To apply the results of the previous section we will consider a suitable Hilbert space of feasible perturbations. In this case we will consider $\mathcal{H}=H^p(\cM, T\cM)$, the Sobolev space of $H^p$ vector fields on $\cM$.
Let us recall the basic notions about the functional space $H^p$. The norm in this case is induced by the inner product $\ip{\cdot,\cdot}_{H^p}$,
\begin{equation} \label{e:hp_product}
\ip{X,Y}_{H^p} := \int_\cM C_0 XY + C_1 DX \cdot DY +\ldots + C_p D^pX \cdot D^pY dx,
\end{equation}
where the dot-product is the sum of element-wise product.
Here, $C_0, \ldots, C_p$ are constants that can be modulated to modify the weight of different orders of derivatives.

By Morrey's inequality part of the Sobolev embedding theorem \cite{Evans2010}, if 
\begin{equation}
\label{e:morrey}
    p \ge 4 + \lfloor \frac M 2 \rfloor  
\end{equation}
then $H^p\subset C^{3}$.
Here $\lfloor \cdot \rfloor$ is the integer part, rounding down to the closest integer below. 
Note that $p$ increases with the ambient dimension $M$. From now on we hence suppose that $p$ is fixed and satisfies \eqref{e:morrey}.

We now formalize the optimal response problem for these perturbations in the framework of Section \ref{sec:gen}. 
We consider $\mathcal{H}=H^p(\cM, T\cM)$; the set of feasible perturbations $P\subseteq H^p$ will be the unit ball of such space and the function $I$ considered in Section \ref{sec:gen} is the inclusion $I:H^p \xhookrightarrow{} C^3$.  
Since $H^p \subset C^3$ we  get that $R: H^p\to \R$ as defined in \eqref{Resp2} is continuous.

The  optimal response problem \eqref{e:gen-func-opt-prob} in this case takes the following form: we search for an optimal perturbation $X_{opt}$ for which
\begin{equation} \begin{split} 
\label{e:define optr}
R(X_{opt}) = \max_{X\in P}R(X),
\quad \textnormal{ where } \quad \\ 
P := \{X\in H^p(\cM): ||X||_{H^p} \le 1 \}.
\end{split} \end{equation}

By Propositions \ref{gen:ex} and \ref{gen:algo} it directly follow 
\begin{proposition} 
\label{p:hpp}
When $R$ is not null the problem \eqref{e:define optr} has a unique solution $X_{opt}$. Furthermore, $X_{opt} = v/||v||_{H^p}$, where $v$ is a representative of $R$ in $H^p$.
\end{proposition}


In the following subsection we see how to choose a suitable Fourier basis to compute $v$, in the case $\cM$ is the $M-$dimensional torus.

\subsection{Expressing \texorpdfstring{$v$}{v} via  Fourier basis when \texorpdfstring{$\cM$}{M} is the torus.}
\hfill\vspace{0.1in}
\label{s:fourier}

In this section  we consider the optimization problem \eqref{e:define optr} in the case where  $\cM  = \T^M$ is the $M$ dimensional torus. 
In this case it is easy to construct an explicit Fourier basis for $H^p$, and we show how to use this to compute the Fourier  approximation for $v$ in this case. 
When $\cM$ is not the torus, finding a basis for $H^p$ adds a further difficulty; or we can opt for solving a PDE for $v$ instead of the Fourier method: this is discussed in \Cref{s:no basis}.



%

In the case $\cM=\T^M$
we have an obvious orthogonal (non-normalized yet) basis of $H^p(\cM)$, which is made by the products of trigonometric functions.
Let us adopt the multi-index notation 
\[\vec{n}:=(n_1, \ldots n_M),\] 
and consider
\begin{equation}\begin{split}
\label{e:basis}
B^j_{\vec{n}}
=
e_j \prod_{i=1}^M b_{n_i}(x_i),\quad
1 \le j\le M,\quad
n_i \ge 0,
\quad \textnormal{where} \quad \\
e_j = (0,0,\ldots,0,1,0,\ldots,0)\in \R^M;
\\
b_m(x_i) = 
\begin{cases}
    1, \quad m=0;
    \\
    \sqrt{2}\sin(\lfloor \frac{m+1}{2} \rfloor 2 \pi x_i) \quad &m \textnormal{ odd}
    ;
    \\
    \sqrt{2}\cos(\lfloor \frac{m+1}{2} \rfloor 2\pi x_i) \quad &m \textnormal{ even}.
\end{cases}
\end{split}\end{equation}
where $\lfloor\cdot\rfloor$ means to round down to closest integer.
Note that $\int_\R b_m^2 dx = 1$.
We define the normalized basis function as
\[
\tilde B^j_{\vec{n}}
:=
\frac{B^j_{\vec{n}}}{||B^j_{\vec{n}}||_{H^p}}.
\]

Here the $H^p$ norm is defined by \Cref{e:hp_product} and it has the following expression (note that it in fact does not depend on $j$),  
\begin{equation}\begin{split}
\label{e:Hpnorm of Fourier}
\ip{B^j_{\vec{n}}, B^j_{\vec{n}}}_{H^p}
= 
\sum_{l=0}^p C_l \sum_{\vec{k}=(1,\ldots,1)\in \R^l }^{(M,\ldots,M)} 
\int \left(\partial_{\vec{k}}
\prod_{i=1}^M b_{n_i}(x_i)\right)^2 dx   
\\
=
\sum_{l=0}^p C_l \sum_{\vec{k}=(1,\ldots,1)\in \R^l }^{(M,\ldots,M)} 
(\lfloor \frac{n_{k_1}+1}{2} \rfloor 2 \pi)^2
\ldots
(\lfloor \frac{n_{k_l}+1}{2} \rfloor 2 \pi)^2.
\end{split}\end{equation}
Note that if $m=0$ then $\lfloor \frac{m+1}{2} \rfloor=0$, so the above formula still applies if one of the directions being differentiated is $b_0\equiv 1$.

Hence, the Fourier expansion of $v$ is 
\begin{equation} \begin{split}
\label{e:vFourier}
v = \sum_{j=1}^M \sum_{{\vec{n}} \ge 0 }
C^j_{\vec{n}} \tilde B^j_{\vec{n}},
\quad \textnormal{where} \quad
C^j_{\vec{n}} 
= \ip{v, \tilde B^j_{\vec{n}}}_{H^p}
= R( \tilde B^j_{\vec{n}}).
\end{split} \end{equation}
Combining the result of Proposition \ref{gen:algo} with the Fourier expression in \Cref{e:vFourier}, we can see that the optimal perturbation achieving the optimal response is
\begin{equation} \begin{split} \label{e:Y}
X_{opt} = \frac{v}{||v||_{H^p}}
= \sum_{j=1}^M \sum_{{\vec{n}} \ge 0 }
\frac{R( \tilde B^j_{\vec{n}})}{||v||_{H^p}} \tilde B^j_{\vec{n}},
\quad \textnormal{where} \quad
||v||_{H^p} = \left( \sum_{j=1}^M \sum_{{\vec{n}} \ge 0 }
(C^j_{\vec{n}})^2 \right)^{\frac12}.
\end{split} \end{equation}

In numerical computations, the number of Fourier basis functions grows exponentially fast on $M$, in the case of the $M$-dimensional torus.
More specifically, in numerics, for each $i$, we typically let $n_i$ only take integer values in $[0, N-1]$, so $\vec{n} \in [0, N-1]^M$.
Also $j$ ranges in $[1,M]$, so the overall number of basis elements is $N_{basis} = M N^M$.
Note that in this case, where we consider a large set of perturbations, acting in all the possible ways on the phace space the number of elements in the basis grows exponentially fast as $M$,
{
so we can not afford real high-dimensional numerical applications if we set $\mathcal{H}=H^p$.
However, as we will show in the following, restricting the space of perturbations we can apply our method to high dimensional phase spaces. 
}

\subsection{Optimal response for additive perturbations and \texorpdfstring{$H^{p}$}{H\^p} representatives in this case
}\label{S:add}
\hfill\vspace{0.1in}

In many cases, perturbations of the dynamics which are natural to be considered are applied directly to the map $f$ defining the dynamics, instead of composing it by a diffeomorphism near the identity as considered in \Cref{s:Hp rep}. 
In this subsection, we prove that for these kinds of perturbations we have results analogous to the results proved in Section \ref{s:Hp rep}.


We  still assume that $f$ is a  $C^3$ uniformly hyperbolic  diffeomorphism with a compact attractor $K$ and $\Phi$ is a $C^3$ observable. We then consider $\overline{\gamma} >0$ and a family of such diffeomorphisms $f_\gamma$, where $\gamma \in [0, \overline{\gamma} )$ such that $f=f_0$.
We also suppose that $\gamma\mapsto f_\gamma$ is $C^1$ from $[0, \overline{\gamma} )$ to the space of  $C^3$ diffeomorphisms. We denote as $\mu_\gamma$ the physical measures of $f_\gamma$ as before, also in this case  the linear response is known to exist.
Considering such  a family of maps $f_\gamma$. The direction of perturbation is then defined as 
\begin{equation*} 
X'
:= \left. \pp{f_\gamma}\gamma \right| _{\gamma=0}.
\end{equation*}
This is still a vector field on $\cM$ but its action on the dynamics is now different from the one of  Section \ref{s:Hp rep}.
We then denote the linear response operator for these kinds of perturbations as
\[\mathcal{R}'(X'):=\delta \mu_\gamma(\Phi).\]

Now we relate this kind of perturbation and response to those considered before.
Since $f$ is a diffeomorphism, we can relate the additive perturbations with the perturbations used in Sections \ref{s:review far} and \ref{s:Hp rep} by defining a diffeomorphism perturbation $\tf_\gamma$ which is equivalent to the additive perturbation by: 
$$\tf_\gamma:=f_\gamma \circ f^{-1}.$$
Denoting as before $X=\delta \tf$ we have 
\begin{equation*} 
X'
= \delta (\tf \circ f)
= \delta \tf \circ f
= X\circ f.
\end{equation*}
Then we can relate the linear response for the additive perturbation $X'$ with the one for the family of diffeomorphisms and set up the related optimization problem.
By \eqref{Resp}  we have:
\begin{equation}\label{sop}
  \mathcal{R}'(X') = \mathcal{R}(X'\circ f^{-1}) = \mathcal{R}(X). 
\end{equation}

We now consider the setting of Section \ref{sec:gen}, with $\mathcal{H}=H^p$ and $P$ being the unit ball of $H^p$ as before, but a vector field there is interpreted as the direction of an additive perturbation by defining for each $X'\in \mathcal{H}$
$$I(X'):=X'\circ f^{-1}.$$
We remark that since $f^{-1}$ is $C^3$, then $I:H^p \to C^3$ is also linear and continuous.

With the notation of Section \ref{sec:gen}, we formulate the optimal response problem for additive perturbations by considering $R'(X'):=\mathcal{R}(I(X'))$ (which in turn is equal to $\mathcal{R'}(X')$ with the notation of \eqref{sop}) and  searching for $X'_{opt}$ such that
\begin{equation}\label{opt-prob-add}
{R'}(X'_{opt})
= \sup_{X' \in P}
{R'}(X') \; 
\end{equation}
where $P := \{X\in H^p(\cM): ||X||_{H^p} \le 1 \}$.




By Proposition \ref{gen:ex} we hence have that the problem has unique solution when $R$ is not constantly equal to $0$. 
Furthermore, Proposition \ref{gen:algo} gives a way to compute this unique solution.

\begin{proposition} When $R'$ is not null the problem \eqref{opt-prob-add} has a unique solution $X'_{opt}$. Furthermore, $X'_{opt} = v'/||v'||_{H^p}$, where $v'$ is a representative of $R'$ in $H^p$.
\end{proposition}

The Fourier scheme described in \Cref{s:fourier} also applies to $v'$, so we can approximate the optimal perturbation $X'_{opt}$ by computing its Fourier coefficients with respect to a suitable basis.

In Section \ref{s:numeric}, we show examples of computation of the optimal perturbation in this additive perturbation case.
In the examples, we will explicitly write how the perturbation is applied to the dynamics.

\section{Algorithms and numerical examples}
\label{s:numeric}


In this section we implement an algorithm for the numerical approximation of the optimal perturbation $X_{opt}$ on a torus and show some examples of such a computation.
First we show two low dimensional examples, in dimension 2 and 3, where we can visualize the attractor and the vectors fields, then a higher dimensional example whose dimension is 21.

In our approximation the Fourier coefficients for  $v$ is are linear responses of Fourier basis, which are computed by the fast adjoint response algorithm.
This algorithm computes the shadowing contribution of the linear response by the nonintrusive adjoint shadowing algorithm, and the unstable contribution by the equivariant divergence formula (see \cite{far} for details on the fast adjoint response algorithm).
The marginal cost for computing $R(X)$ of a new $X$ is much smaller than the first,
since the main computations are moved away from $X$.
It is more efficient for our current situation than the `tangent' version of the fast response algorithm in \cite{fr}, since here we want to compute $R( B^j_{\vec{n}})$ for many $j$ and $\vec{n}$.

The fast response algorithm has two main errors.
The first is due to using a finite decorrelation step number $W$ (see e.g. \eqref{R2W}), and the scale of the error is $  h\sim O(\theta^W)$ for some $0<\theta<1$, related to the speed of decay of correlation in the system.
The second error is the sampling error due to using finite orbit length $T$, like most Monte-Carlo algorithms, the scale of the error is $O(\sqrt{W/T})$.
Using these two estimations, we gave the rough cost-error estimation of the fast response algorithm in \cite{fr}, and studied the numerical convergence with respect to $T$ and $W$ in \cite{fr,far}.

Lastly, we need to design some smart strategy for enumerating all $\vec{n}$'s; such strategy should work for any dimension $M$.
This is not a mathematical question but it did posed significant difficulty for designing the algorithm.
This is explained in \Cref{s:enum vecn}.

\subsection{Numerical example: 2d solenoid-like map}
\hfill\vspace{0.1in}

This subsection illustrates our algorithm on a map on $\cM = \T^2$ with unstable dimension $u=1$ and stable dimension $s=1$.
The base dynamical system here considered is
\[ \begin{split}
  x^1_{k+1} &=f^1(x_n):= 0.5 x^1_{k} + 0.01  \cos(2\pi x^2_{k}) \\
  x^2_{k+1} &=f^2(x_k):= 2x^2_{k} + 0.1 x^1_{k} \sin(2\pi x^2_{k}) \mod 1,
\end{split} \]
where the superscript labels the coordinates, and $k$ labels the time-step.
Note that here we choose the first coordinate to be $[-0.5,0.5]$, the second to be $[0,1]$ for convenience.
The observable function is 
\[ \begin{split}
  \Phi(x) := (x^1)^3 + 0.5 (x - 0.5)^2.
\end{split} \]
\Cref{f:orbit} shows a typical orbit of this nonlinear system, which indicates that the attractor is fractal.

\begin{figure}[ht]
    \centering
    \includegraphics[width=0.5\linewidth]{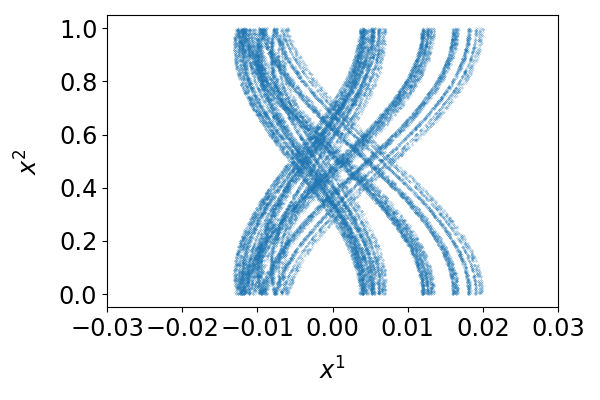}
    \caption{The scatter plot of a typical orbit.}
    \label{f:orbit}
\end{figure}

For convenience, we define the perturbed dynamics as 
\begin{equation}
\label{e:f ga}
   f_\gamma:= f+\gamma X'.
\end{equation}
The space of feasible infinitesimal perturbations is
\[
P:=\{||X'||_{H^5} = 1\}\subset \cH: = H^5(\cM,T\cM) \subset C^{3}(\cM,T\cM).
\]
The weight coefficients in the $H^p$ norm defined in \Cref{e:Hpnorm of Fourier} are chosen as $C_l = (2\pi)^{-2l}$, so
\[\begin{split}
\ip{B^j_{\vec{n}}, B^j_{\vec{n}}}_{H^p}
= 
\sum_{l=0}^p \sum_{\vec{k}=(1,\ldots,1)\in \R^l }^{(M,\ldots,M)} 
(\lfloor \frac{n_{k_1}+1}{2} \rfloor )^2
\ldots
(\lfloor \frac{n_{k_l}+1}{2} \rfloor )^2
,\\
\quad\textnormal{for}\quad
0\le n_i \le N-1=14,
\quad
1\le j\le M=2.
\end{split}\]
Here $N=15$ is the number of basis in each direction of $x$.
The contour plot of the norms of all the non-normalized basis are in \Cref{f:B2contour}.
With this, we can normalize all the basis to get $\tilde B^j_{\vec{n}}$.

\begin{figure}[ht]
    \centering
    \includegraphics[width=0.5\linewidth]{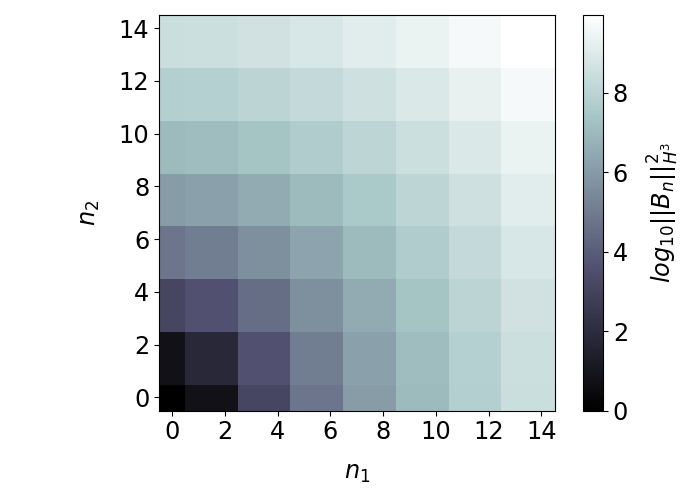}
    \caption{Contour plot of $||B^j_{\vec{n}}||^2_{H^5}$ in log scale.}
    \label{f:B2contour}
\end{figure}

We use the fast adjoint response algorithm to compute the linear response of orthonormal basis, $R'(\tilde B^j_{\vec{n}})$, which is also the coefficient for $v$, the $H^p$ representative of the linear response operator $R'$.
The results are plotted in \Cref{f:RBcontour}.

The default setting for the fast adjoint response algorithm, $N_{seg}=20$ steps in each segment, 
$A=4000$ segments, and $W=10$, is used unless otherwise noted.
The code is at \url{https://github.com/niangxiu/optr}.
On a 3 GHz single core computer, the computation time for computing the linear responses of all $2\times15\times15=450$ basis is 863 seconds.

\begin{figure}[ht]
    \centering
    \includegraphics[width=0.49\linewidth]{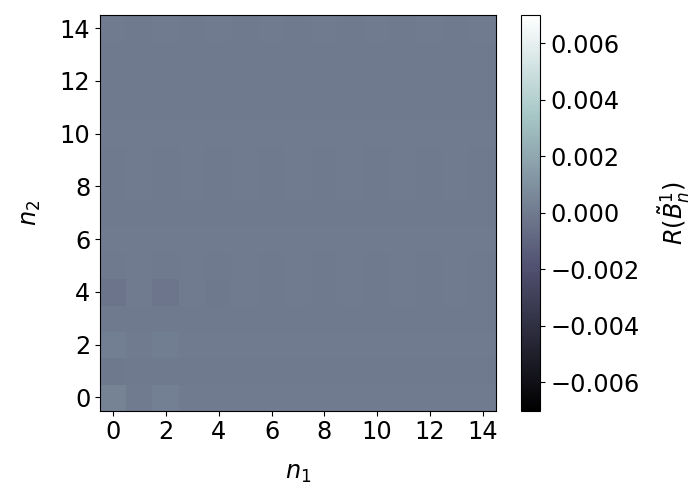}
    \includegraphics[width=0.49\linewidth]{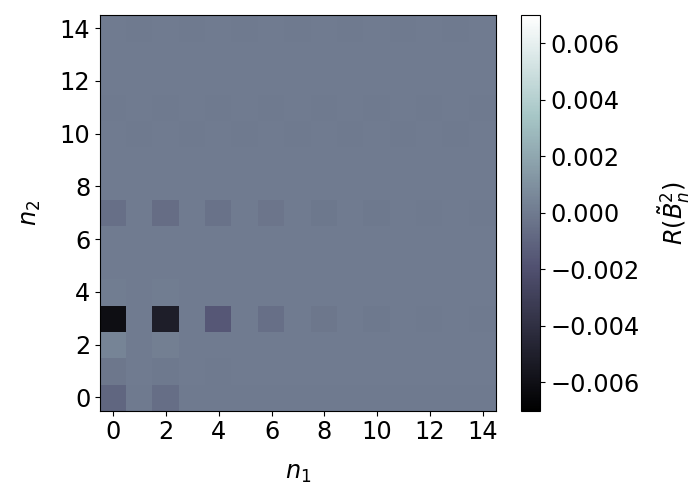}
    \caption{Contour plot of $C^j_{\vec{n}} = R'( \tilde B^j_{\vec{n}})$. 
    Left: $j=1$. Right: $j=2$.}
    \label{f:RBcontour}
\end{figure}

With $R'(\tilde B^j_{\vec{n}})$, we can compute $v'$, then compute $X'_{opt}$, the optimal perturbation achieving the optimal response in \Cref{opt-prob-add}.
The vector field plot of $X_{opt}$ is in \Cref{f:Yvec}.
{
The first six Fourier coefficients of $X'_{opt}$, for  the basis functions $\tilde B^1_{(0,0)},\ldots, \tilde B^1_{(0,5)}$, are:
\begin{equation*}
\begin{split}
[c_0,...,c_{5}]& =
[\texttt{4.4e-2, -8.0e-3, 2.0e-2, -7.8e-4, -3.6e-2, -4.2e-5}].
\end{split}
\end{equation*}
Here \texttt{e-3} means to multiply by $10^{-3}$.
The largest (in absolute value) Fourier coefficient is \texttt{-0.73}, attained for $\tilde B^2_{(0,3)}$, which generates the largest linear response among all the basis.
}

\begin{figure}[ht]
    \centering
\includegraphics[width=0.6\linewidth]{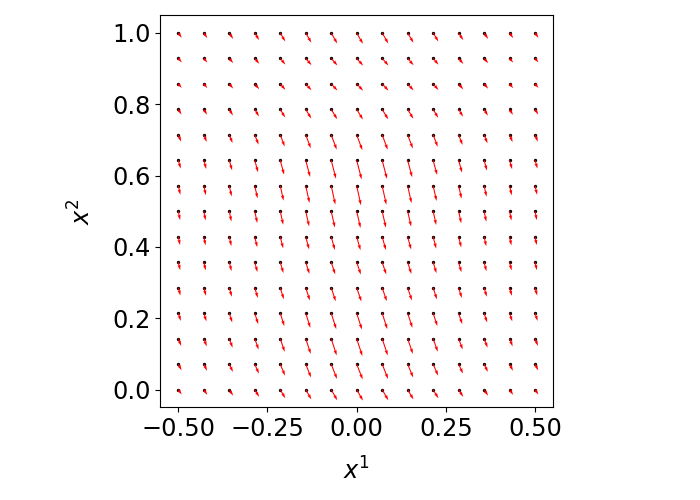}
    \caption{Vector field plot of the optimal perturbation $\frac 14 X'_{opt}$.}
    \label{f:Yvec}
\end{figure}

\Cref{f:prm obj} verifies that the optimal perturbation $X'_{opt}$ computed by our method indeed generates the optimal response.
The optimal response has larger absolute value than the response of any individual basis function.
In particular, it is larger than the linear response of $\tilde B^2_{(0,3)}$, which, by \Cref{f:RBcontour}, generates the largest linear response among all basis.
We also plot the linear response of $\tilde B^2_{(14,14)}$.

Moreover, we verify that the linear response we compute is correct in reflecting the trend between $\gamma$ and $\mu_\gamma(\Phi)$.
To do this, we compute $\mu_\gamma(\Phi)$ for different values of $\gamma$ around 0 and different perturbations, $X'_{opt}$, $\tilde B^2_{(0,3)}$, and $\tilde B^2_{(14,14)}$.
The trend matches the linear response we compute.

\begin{figure}[ht]
    \centering
    \includegraphics[width=0.5\linewidth]{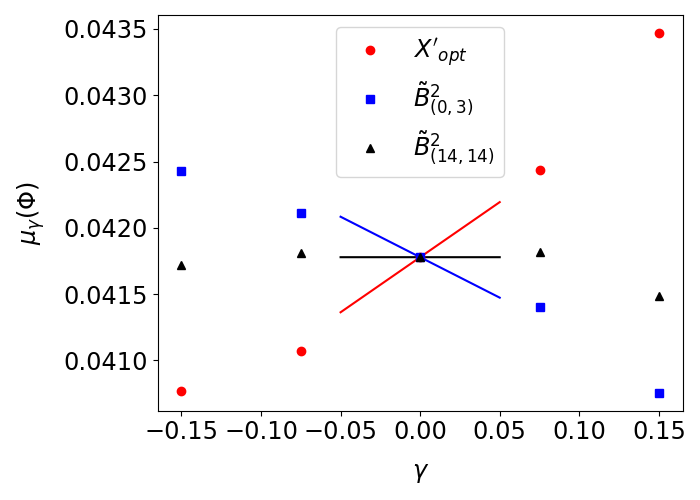}
    \caption{Linear responses and $\mu_\gamma(\Phi)$ of different perturbations.
    The dots are $\mu_\gamma(\Phi)$ for $f+\gamma X'_{opt}$ (indicated by red circles), $f+\gamma \tilde B^2_{(0,3)}$ (blue squares),
    and $f+\gamma \tilde B^2_{(14,14)}$ (black triangles).
    The short lines represent the linear responses computed by the fast response algorithm.
    }
    \label{f:prm obj}
\end{figure}

\subsection{Numerical example: 3d solenoid-like map}
\hfill\vspace{0.1in}

We apply our method on a dynamical system in $\cM = \T ^3$ similar to the map in the previous subsection but with increased dimension.
The base governing equation is 
\[ \begin{split}
  x^1_{n+1} &= f^1(x_n):=  0.5x^1_{n} + 0.01 \sum_{i = 2}^{3} \cos(2\pi x^i_{n})\\
  x^i_{n+1} &= f^i(x_n):= 2x^i_{n} + 0.1 x^1_{n} \sin(2\pi x^i_{n}) \mod 1, \quad \textnormal{for} \quad 2\le i\le 3.
\end{split} \]
And the observable function is 
\[ \begin{split}
  \Phi(x) := (x^1)^3 + 0.5 \sum_{i = 2}^{3} (x^i - 0.5 )^2.
\end{split} \]
The space of feasible infinitesimal perturbations is
\[
P:=\{||X'||_{H^5} = 1\}\subset \cH: = H^5(\cM,T\cM) \subset C^{3}(\cM,T\cM).
\]

When $\cH = H^5(\cM,T\cM)$, due to rapidly growing (with respect to $M$ and $N$) computational cost, we reduce the number of basis in each direction to
\[  
N = 11.
\]
So now we have $3\times 11^3=3993$ functions in the Fourier basis.
The wall-clock time to run the fast adjoint response algorithm on an orbit, with other setting the same as the previous subsection, and with the same computer, is about 5 hours.
The computational cost is much higher than the 2d case, mainly because the number of basis functions is about 10 times as before, and also because $M$ is 1.5 times as before.
Moreover, for this specific model, $u$ is twice as before, and the Jacobian is almost dense, so the cost grows as $O(MN^M uM^2)=O(uN^M M^3)$.
For many real-life applications,  $u$ is fixed regardless of increasing $M$, and the Jacobian is sparse, so the cost would scale as $O(N^M M^2)$ when $\cH$ is the $H^p$ vector fields over $\cM$; this is very expensive.

\begin{figure}[ht]
    \centering
\includegraphics[width=0.49\linewidth]{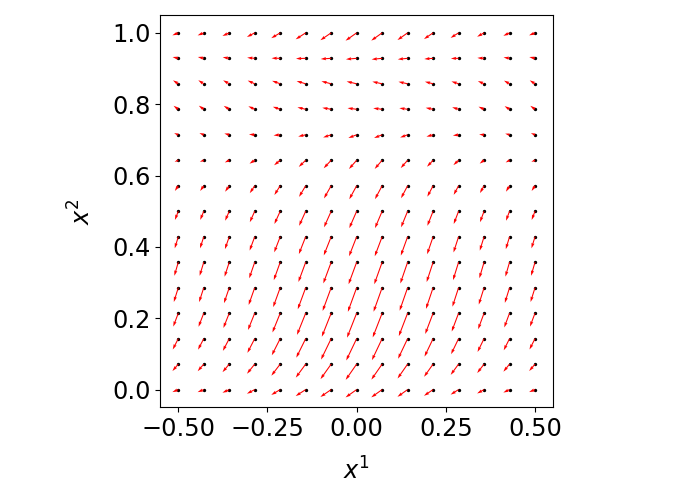}
\includegraphics[width=0.49\linewidth]{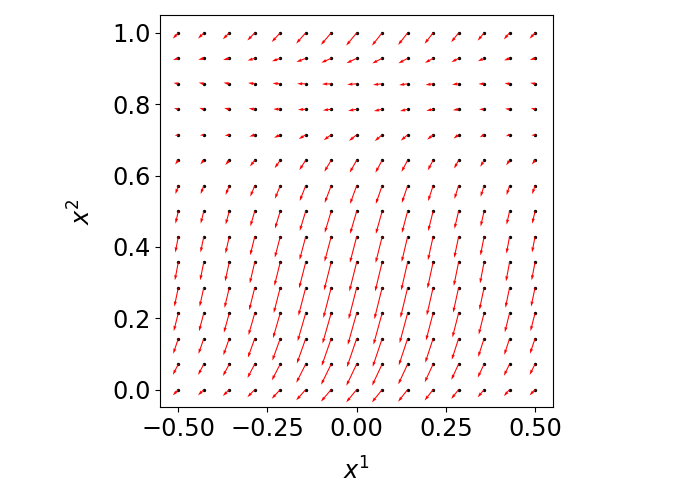}
\caption{Vector field plots of $ \frac 1{4} [X'^1_{opt}, X'^2_{opt}]$ for the 3-dimensional system. 
Left: slice at $x^3=0$.
Right: slice at $x^3=0.5$.
    }
    \label{f:Xopt 3d}
\end{figure}

\begin{figure}
    \centering
    \includegraphics[width=0.5\linewidth]{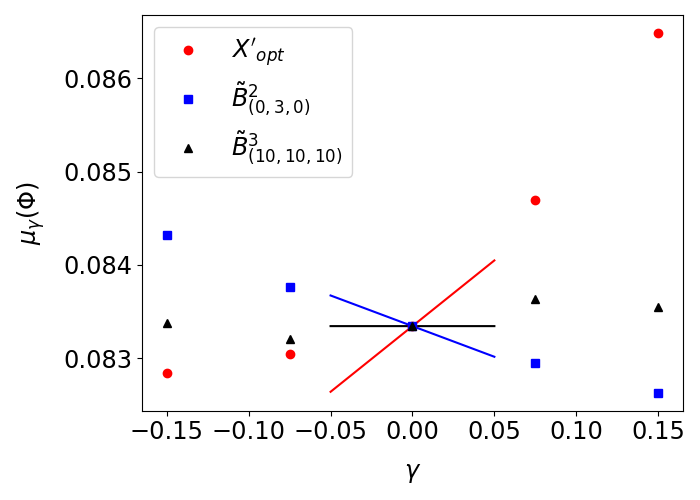}
    \caption{Linear responses and averaged observable of different perturbations of a 3d dynamical system (explanations in the caption of Figure \ref{f:prm obj}).}
    \label{f:prm obj 3d}
\end{figure}

We plot slices of the vector fields $X'_{opt}$ in \Cref{f:Xopt 3d}.
{
The first  six Fourier coefficients of $X'_{opt}$, relative to  the basis functions $\tilde B^1_{(0,0,0)}, \ldots, \tilde B^1_{(0,0,5)}$, are:
\begin{equation*}
\begin{split}
[c_0,...,c_{5}]& =
[\texttt{-7.3e-2, -5.6e-3, -5.8e-3, 4.1e-4, -2.2e-2, -1.3e-4}].
\end{split}
\end{equation*}
The largest (in absolute value) Fourier coefficient is \texttt{-0.47}, attained for $\tilde B^2_{(0,3,0)}$, which generates the largest linear response among all the basis.
}
Then we plot $\mu_\gamma(\Phi)$ versus $\gamma$ and the linear responses at $\gamma=0$ computed for different perturbations in \Cref{f:prm obj 3d}.
This shows that the optimal perturbation indeed generates a linear response larger than all elements in the basis.
It also verifies that the linear response we compute is correct.

\subsection{Numerical example: 21-dimensional solenoid-like map with a "small" \texorpdfstring{$\cH$}{H}}
\hfill\vspace{0.1in}
\label{s:lowH}

We give another example where the phase space is high-dimensional, but $\cH$ is the Sobolev space on a line.
More specifically, the base dynamics in $\T ^{21}$ is
\[ \begin{split}
  x^1_{n+1} &= f^1(x_n):= 0.1x^1_{n} + 0.01 \sum_{i = 2}^{21} \cos(2\pi x^i_{n})\\
  x^i_{n+1} &= f^i(x_n):= 2x^i_{n} + 0.1 x^1_{n} \sin(2\pi x^i_{n}) \mod 1, \quad \textnormal{for} \quad 2\le i\le 21.
\end{split} \]
The unstable dimension is $u=20$.
And the observable function is 
\[ \begin{split}
  \Phi(x) := x^1 + 2 \sum_{i = 2}^{21} (x^i - 0.5 )^2.
\end{split} \]

The perturbed dynamics is
\[
f_\gamma:= f+\gamma X',
\]
where $X'\in  \cH$.
We set $\cH$ as
\[
\cH:=
\{h(x):h^1(x)= h^2(x)=g(x^1), \,
h^3(x) = \ldots = h^{21}(x)=0, \,
g \in H^4(\R)\}.
\]
We set the norm $||h||_\cH = ||h^1(x^1)||_{H^4(\R)}$, 
so $\cH$ is isomorphic to the Sovolev space $H^4$ on the real line.
Hence, $\cH\subset C^3(\T^{21})$ and the inclusion map is continuous.
We set the feasible set $P$ as the unit ball of $\cH$.

The unormalized Fourier basis for $\cH$ is
\[
B_n(x) = [b_n(x^1),b_n(x^1),0,\ldots,0] \,,
\]
where $b_n$ is the 
{
trigonometric 
}
function given in \Cref{e:basis}.
In numerical computations, we truncate the Fourier basis to
\[  
0\le n\le N_{basis} = N = 21.
\]
The wall-clock time to run the fast adjoint response algorithm on an orbit, with other setting the same as the previous subsection, and with the same computer, is about 110 seconds.
The cost 
{
of computing linear responses of all basis functions and the optimal response
}
is much lower than the 3d case because the number of basis functions needed is smaller.

\begin{figure}[ht]
\centering
\includegraphics[width=0.6\linewidth]{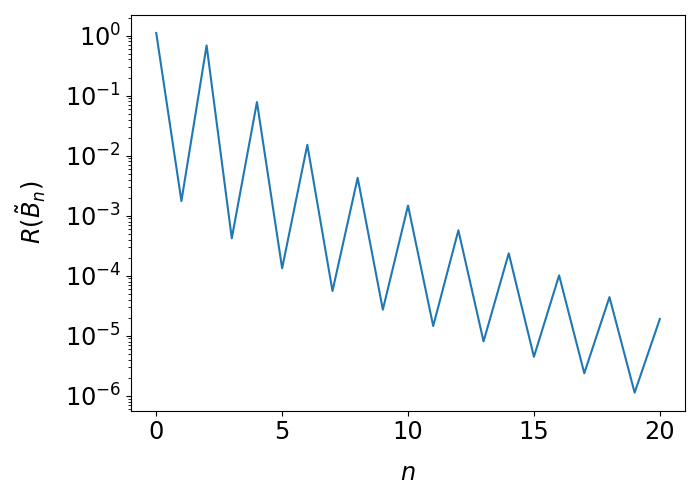}
\caption{$ R(\tilde B_n)$, the linear response for normalized basis functions.
}
\label{f:RB lowH}
\end{figure}

The linear response for each basis function is plotted in \Cref{f:RB lowH}, which basically decays exponentially as $n$ increases, so the error caused by using a finite basis is also very small.
The first seven Fourier coefficients of $X'_{opt}$ with respect to the basis $\tilde B_n$ are:
\begin{equation*}
\begin{split}
[c_0,...,c_{6}]& =
[\texttt{8.5e-1, 1.3e-3, 5.3e-1, 3.2e-4, 6.0e-2, 1.0e-4, 1.2e-2}].
\end{split}
\end{equation*}

\begin{figure}[ht]
    \centering
\includegraphics[width=0.49\linewidth]{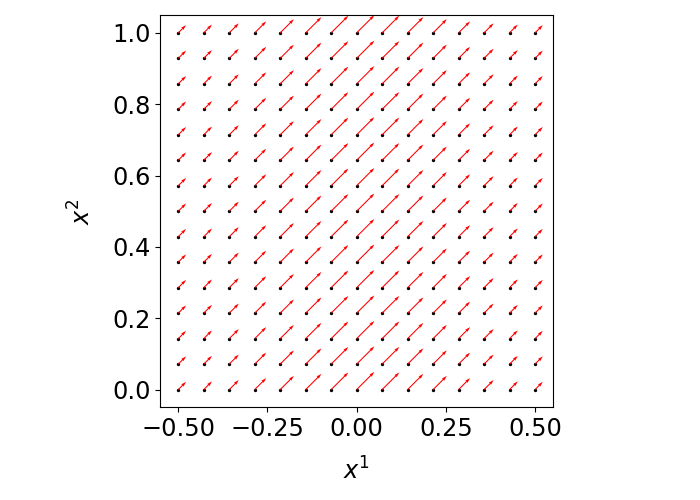}
\includegraphics[width=0.49\linewidth]{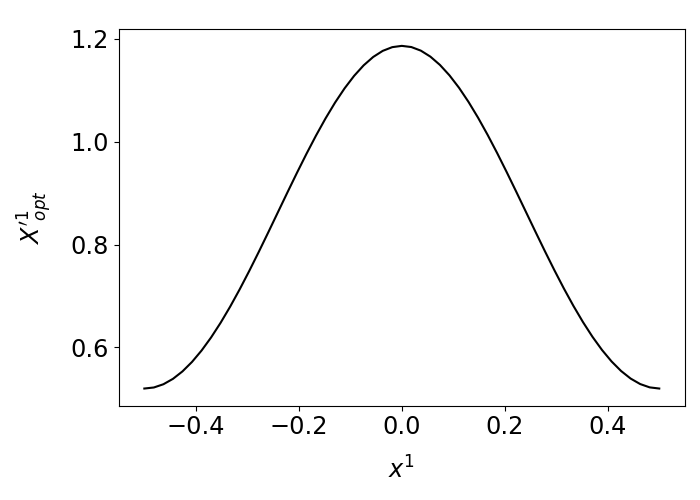}
    \caption{Left: vector field plot of $ \frac 1{24} X'_{opt}$ for the 21-dimensional system. We plot only the first two coordinates, as all the other are $0$.
    {
    Right: $X'^1_{opt}(x^1)$ function plot.
    }
    }
    \label{f:YveclowH}
\end{figure}

\begin{figure}
    \centering
\includegraphics[width=0.5\linewidth]{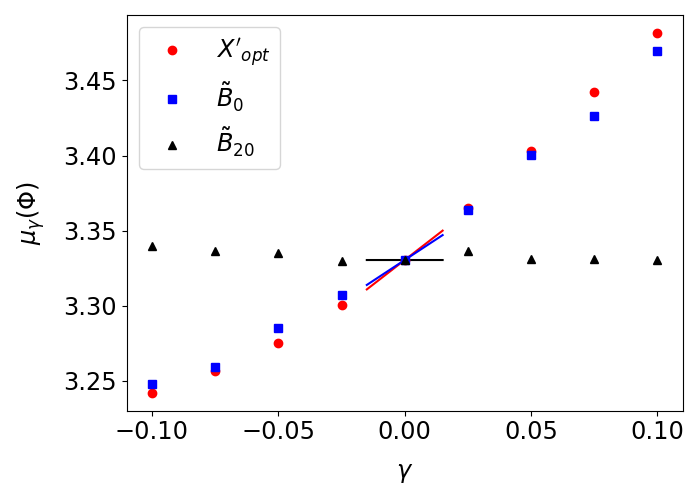}
    \caption{Linear responses and averaged observable of different perturbations of a 21d dynamical system (explanations in the caption of Figure \ref{f:prm obj}).}
    \label{f:prm obj lowH}
\end{figure}

The vector field plot of the optimal perturbation $X_{opt}$ is in \Cref{f:Yvec}.
We can see that $X^1_{opt}=X^2_{opt}$ and it depends only on $x^1$; this is due to the selection of $\cH$.
Then we plot $\mu_\gamma(\Phi)$ versus $\gamma$ and the linear responses at $\gamma=0$ computed for different perturbations in \Cref{f:prm obj lowH}.
Here, the largest linear response among all the basis elements is generated by $\tilde B_{0}$.
This shows that the optimal perturbation indeed generates a linear response larger than all elements in the basis.
This also verifies that the linear response we compute are correct.
Due to the relatively high dimension of $\cM$ and its unstable space, it would be difficult to compute the optimal perturbation for this example using  a  finite-element reduction of the transfer operator associated to the system,  rather than the fast response algorithm.

\section{When \texorpdfstring{$\cM$}{M} is not a torus.}
\label{s:2 general cases}
\label{s:no basis}

In this section, we discuss the generalization of the results of Section \ref{s:exist and char} to the case where $\mathcal{H}=H^p$, but $\cM$ is not a torus.
Rather, $\cM$ is a compact region of $\mathbb{R}^M$ with a smooth boundary.
We will not perform numerical experiments for this case.

The case we are considering now is different from the torus, since we no longer have an obvious basis of $H^p(\cM)$.
Instead, we can solve the Laplacian equation to obtain an expression of $v$, which is the $H^p_0$ representative of the linear response operator $R\in H^{-p}$.
For convenience, in this subsection we assume that the coefficients in \Cref{e:hp_product} are $C_0 = 1, \ldots, C_p=1$.

Consider a compact subset $\cM$ of $R^M$ with $C^p$-smooth boundary.
Assume that the support of the physical measure is in $\cM$.
Denote 
\[
H^p_0 :=\{X\in H^p(\cM):
X|_{\partial \cM} = \partial_n X|_{\partial \cM}
=\ldots = \partial_n^{(p-1)} X|_{\partial \cM} = 0\},
\]
where $\partial_n$ is the derivative along the normal direction of the boundary $\partial \cM$.
By the same arguments as theorem 1 in \cite[section 5.5]{Evans2010}, since boundary is $C^p$ and $X\in H^p$, the boundary value of $X$ is well-defined in the trace sense, and is in $H^{p-1}(\partial \cM)$.
The optimal response problem is 
\[\begin{split}
\max_{X \in H_0^p(\cM)} R(X)
\;, \quad
\textnormal{s.t. } ||X||_{H^p} = 1 \;.
\end{split} \]
Note that the boundedness of $R$ on $H^p(\cM)$ established in \Cref{s:Hp rep} is still effective in this case.
Hence, $R$ is also bounded on $H^p_0(\cM)$, and there is a unique representative of $R$ in $H^p_0(\cM)$, denoted by $v$.

Define the differential operator $A$ by
\[\begin{split}
   AX: = X-\Delta X + \ldots + (-\Delta)^p X.
\end{split}\]
Note that $A$ reduces the order of differentiability by $2p$.
Then, for $X\in C^\infty\bigcap H^p_0$,
\[\begin{split}
    \ip{X,AX}_{L^2} = ||X||_{H^p}^2
    = \int_\cM X^2 + |DX|^2 +\ldots +|D^p X|^2 dx \\
    = \int_\cM X\left( X-\Delta X + \ldots + (-\Delta)^p X \right)
    = \ip{X,X}_{H^p}
    .
\end{split}\]
Note that here the $(p-1)$th derivative of the boundary value of $X$ is a well-defined function, so we can integrate by parts.
Moreover, $\ip{X,AY}_{L^2}= \ip{X,Y}_{H^p}$ if $Y$ is also in $\in C^\infty\bigcap H^p_0$.

By Lax-Milgram theorem \cite[section 6.2.1]{Evans2010}, the following PDE
\[
Av = R, \quad R\in H^{-p}(\cM)
\]
has a unique weak solution, denoted as $A^{-1}R:=v\in H_0^p(\cM)$, that is, 
\[
\ip{v,X}_{H^p} = R (X),
\quad\textnormal{for any}\quad
X\in H_0^p(\cM).
\]
Hence, as previous sections, $v$ is the $H^p_0$ representative of $R\in H^{-p}$.

To numerically compute $v$, we need to first compute an approximation of the `$L^2$ representative' of $R$ by some kind of finite element method.
First, in the finite element method, we compute an approximate mollified density of the physical measure, which is now a differentiable function, not a distribution.
Then, by integration by parts (in the discrete context, this is just collecting terms with $X$ evaluated at the same point together), we can find a function $r$, such that $R(X)\approx \int_\cM r X dx$.

Then we can find $ \tilde v $, the approximate $H^p_0$ representative of $R$.
More specifically, for any $C^\infty_0$ test function $X$, we want
\[
\ip{ \tilde v ,X}_{H^p} = \int_\cM r X dx \approx R (X).
\]
By integration by parts, the left hand side is
\[
\int_\cM  \tilde v X + D \tilde v  \cdot DX +\ldots + D^p \tilde v  \cdot D^pX dx
= 
\int_\cM ( \tilde v -\Delta  \tilde v  + \ldots + (-\Delta)^p  \tilde v ) X dx
\]
By substitution, we want the following holds for any $X\in C^\infty_0$,
\[
\int_\cM ( \tilde v -\Delta  \tilde v  + \ldots + (-\Delta)^p  \tilde v ) X dx = \int_\cM r X dx.
\]
Hence, we can solve $ \tilde v $ by numerically solving the high-order Laplacian equation $A \tilde v  = r$.
More specifically, the equation to be solved is
\begin{equation} \begin{cases}
 \tilde v -\Delta  \tilde v  + \ldots + (-\Delta)^p  \tilde v = r, 
\quad \textnormal{ in } \cM ; \\
 \tilde v  = \partial_n  \tilde v  =\ldots = \partial_n^{(p-1)}  \tilde v =0, 
\quad \textnormal{ on } \partial \cM.
\end{cases} \end{equation}
Note that this PDE is in the strong sense, since now $r$ is a function rather than a distribution.

We may write an integral formula of $ \tilde v $ by Green's function.
In particular, denote $G(x,\cdot)\in H_0^p(\cM)$ as the Green's function, which is the weak solution of
\[
A (G(x,y)) = \delta_x(y),
\quad
x\in \cM.
\]
Here $\delta_x$ is a bounded linear functional on $H^p_0$.
Note that here $A$ is a differential operator in $y$.
Using $G$, we can write the explicit expression of $A^{-1}r$,
\[
 \tilde v (y) = (A^{-1}r) (y) 
= \int_\cM G(x, y) r(x) dx.    
\]
This integral formula is nice to look at, but, for numerical purposes, it might not be as useful as directly solving the Laplacian equation.

\appendix

\section{Enumerating \texorpdfstring{$\vec{n}$}{n} in computer programming}
\label{s:enum vecn}

The computer program needs to enumerate over all basis functions, $B_{\vec{n}}^j$.
In this section we adopt \textit{python} numbering, where all sequences starts from the zeroth element.
The most straight forward idea is to construct  $M+1$ levels of for-loops, one level for enumerating over $0\le j\le M-1$, then one level for each $0\le n_i\le N-1$.
However, this would require changing the number of for-loop levels, hence changing the code, for each different $M$.

This section explains how to enumerate over all $B_{\vec{n}}^j$ so that the code works for any $M$.
The first function (in computer programming sense) transforms an integer $m$ into the vector $(j, \vec{n})$, such that
\[
m = j N^M + \sum_{i=0}^{M-1} n_i N^i.
\]
This function, listed in \Cref{alg:jn}, is somewhat standard.
With this function, we only need a single level of for-loop for $m$ from $0$ to $(MN^M-1)$ to enumerate all $B_{\vec{n}}^j$'s.

\begin{algorithm}
  \caption{function int2vec: get $j, \vec{n}$ from a single integer $m$.}
  \label{alg:jn}
\begin{algorithmic}[1]
  \Require $m, N, n_{bits}$
  \State $digits \gets [ \;]$ 
  \Comment{Initialize $digits$ to empty sequence.}
  \For {$i\gets$ 0 to ($n_{bits}$-1) }
  \Comment{Only $(n_{bits}-1)$ many digits are base $N$.}
    \State Insert $(m \% N)$ to the start of $digits$.
    \State $m \gets\lfloor m / N \rfloor$.
    \Comment{$\lfloor \cdot \rfloor$ rounds down to integer.}
  \EndFor
  \State Insert $m$ to the start of $digits$.
 \Ensure $digits$
\end{algorithmic}
\end{algorithm}

\Cref{alg:fgax} shows how to compute $\partial_{x_i} \tilde B^j_{\vec{n}}(x)$ for all $i, j, \vec{n}$, where $0\le i \le M-1$.
Here $x = (x_0, \ldots, x_{M-1})$.
This function generates the $\nabla X$ used in the equivariant divergence formula.
This also gives an example of using \Cref{alg:jn} to reduce the number of for-loops levels.

\begin{algorithm}
  \caption{getfgax: compute $\partial_{x_i} \tilde B^j_{\vec{n}}(x)$ for all $i, j, \vec{n}$ at a given $x$.}
  \label{alg:fgax}
\begin{algorithmic}[1]
  \Require $x = (x_0, \ldots, x_{M-1})$.
  \State $n_{pert} \gets MN^M$.
  \Comment{number of all basis functions.}
  \State $\partial_x\partial_\gamma f \gets$ zero array of size $[n_{pert}, M, M]$.
\For {$m \gets 0, n_{pert}-1$}
\State $j,\vec{n} \gets$ int2vec$(m, N, M+1)$
\Comment{$j$ takes first digit, $\vec{n}$ takes the rest digits.}
\For{$i \gets 0, M-1$} 
\Comment{$i$ labels the direction to take derivative.}
\State $\partial_x\partial_\gamma f [m,j,i] \gets 1$

\For{$k \gets 0, M-1$} 
\Comment{$k$ labels the direction of trigonometric functions.}
  \If{$k == i$}
  \If{$n_k$ is odd}
  \State 
  $\partial_x\partial_\gamma f [m,j,i] \gets 
  \lfloor \frac{n_k}2\rfloor2\pi  \sqrt{2} \cos(\lfloor \frac{n_k}2 \rfloor2\pi  x_k)
  \partial_x\partial_\gamma f [m,j,i]  $
  \Else { $n_k$ is even}
  \State
  $\partial_x\partial_\gamma f [m,j,i] \gets 
  - \lfloor \frac{n_k}2 \rfloor 2\pi \sqrt{2} \sin(\lfloor \frac{n_k}2 \rfloor 2\pi x_k)
  \partial_x\partial_\gamma f [m,j,i]  $
  \EndIf
  \Else
  \If{$n_k==0$ }
  \State do nothing
  \ElsIf {$n_k$ is odd}
  \State 
  $\partial_x\partial_\gamma f [m,j,i] \gets 
  \sqrt{2} \sin(\lfloor \frac{n_k}2\rfloor 2\pi x_k)
  \partial_x\partial_\gamma f [m,j,i]$
  \Else {\;$n_k\neq 0$ is even}
  \State
  $\partial_x\partial_\gamma f [m,j,i] \gets 
  \sqrt{2} \cos(\lfloor \frac{n_k}2\rfloor 2\pi x_k)
  \partial_x\partial_\gamma f [m,j,i]$
  \EndIf
    
\EndIf
\EndFor
\State 
$\partial_x\partial_\gamma f [m,j,i] \gets 
  \partial_x\partial_\gamma f [m,j,i] 
  \,/\,
  ||\partial_x\partial_\gamma f [m,j,i]||_{H_p}$
  
\EndFor
\EndFor
\Ensure $\partial_x\partial_\gamma f $, which is an array of size $[n_{pert}, M, M]$
\end{algorithmic}
\end{algorithm}

\noindent \textbf{Acknowledgments.}  SG is partially supported by the
research project ``Stochastic properties of dynamical systems" (PRIN 2022NTKXCX) of the Italian Ministry of Education and Research.
\newline

\noindent \textbf{Data availability.} Complete descriptions of numerical computations are included in the manuscript text.
\newline

\noindent \textbf{Conflicts of interests or competing interests.}  This submission raises  no conflicts of interests or competing interests for the authors.

\bibliographystyle{abbrv}
{\footnotesize\bibliography{library}}

\end{document}